\documentclass[runningheads,10pt]{llncs}
\usepackage{amsmath,amssymb,amsfonts}
\usepackage{pgfplots}
\usepackage{quiver}

\usepackage{tikzit}
\tikzstyle{observation}=[text=blue]
\tikzstyle{morphism}=[fill=white, draw=black, shape=rectangle]
\tikzstyle{medium box}=[fill=white, draw=black, shape=rectangle, minimum width=0.8cm, minimum height=0.9cm]
\tikzstyle{large morphism}=[fill=white, draw=black, shape=rectangle, minimum width=1cm, minimum height=2cm]
\tikzstyle{bn}=[fill=black, draw=black, shape=circle, inner sep=1.5pt]
\tikzstyle{wn}=[fill=white, draw=black, shape=circle, inner sep=1.5pt]
\tikzstyle{nn}=[fill=gray, draw=gray, regular polygon, regular polygon sides=3, minimum width=0.3cm, shape border rotate=90, inner sep=0pt]
\tikzstyle{state}=[fill=white, draw=black, regular polygon, regular polygon sides=3, minimum width=0.8cm, shape border rotate=90, inner sep=0pt]
\tikzstyle{effect}=[fill=white, draw=black, regular polygon, regular polygon sides=3, minimum width=0.8cm, shape border rotate=0, inner sep=0pt]
\tikzstyle{medium state}=[fill=white, draw=black, regular polygon, regular polygon sides=3, minimum width=1.3cm, inner sep=0pt, shape border rotate=180]
\tikzstyle{medium effect}=[fill=white, draw=black, regular polygon, regular polygon sides=3, minimum width=1.3cm, inner sep=0pt, shape border rotate=0]
\tikzstyle{large state}=[fill=white, draw=black, regular polygon, regular polygon sides=3, minimum width=2.2cm, shape border rotate=180, inner sep=0pt]
\tikzstyle{treenode}=[fill=white, draw=none, shape=circle]

\tikzstyle{arrow}=[->]
\tikzstyle{dashed box}=[-, dashed]
\tikzstyle{condition}=[draw=blue, dashed]
\tikzstyle{multiple}=[-, line width=0.5mm]

\pgfdeclarelayer{edgelayer}
\pgfdeclarelayer{nodelayer}
\pgfsetlayers{edgelayer,nodelayer,main}
\tikzstyle{none}=[]
\tikzstyle{observation}=[text=blue]
\tikzstyle{morphism}=[fill=white, draw=black, shape=rectangle]
\tikzstyle{medium box}=[fill=white, draw=black, shape=rectangle, minimum width=0.8cm, minimum height=0.9cm]
\tikzstyle{large morphism}=[fill=white, draw=black, shape=rectangle, minimum width=1cm, minimum height=2cm]
\tikzstyle{bn}=[fill=black, draw=black, shape=circle, inner sep=1.5pt]
\tikzstyle{wn}=[fill=white, draw=black, shape=circle, inner sep=1.5pt]
\tikzstyle{nn}=[fill=gray, draw=gray, regular polygon, regular polygon sides=3, minimum width=0.3cm, shape border rotate=90, inner sep=0pt]
\tikzstyle{state}=[fill=white, draw=black, regular polygon, regular polygon sides=3, minimum width=0.8cm, shape border rotate=90, inner sep=0pt]
\tikzstyle{effect}=[fill=white, draw=black, regular polygon, regular polygon sides=3, minimum width=0.8cm, shape border rotate=0, inner sep=0pt]
\tikzstyle{medium state}=[fill=white, draw=black, regular polygon, regular polygon sides=3, minimum width=1.3cm, inner sep=0pt, shape border rotate=180]
\tikzstyle{medium effect}=[fill=white, draw=black, regular polygon, regular polygon sides=3, minimum width=1.3cm, inner sep=0pt, shape border rotate=0]
\tikzstyle{large state}=[fill=white, draw=black, regular polygon, regular polygon sides=3, minimum width=2.2cm, shape border rotate=180, inner sep=0pt]
\tikzstyle{treenode}=[fill=white, draw=none, shape=circle]

\tikzstyle{arrow}=[->]
\tikzstyle{dashed box}=[-, dashed]
\tikzstyle{condition}=[draw=blue, dashed]
\tikzstyle{multiple}=[-, line width=0.5mm]

\tikzset{baseline=(current  bounding  box.center)}

\newcommand{\Sec}{\S}

\newcommand{\R}{\mathbb R}
\newcommand{\keyword}[1]{\mathrm{#1}}
\newcommand{\catname}[1]{\mathtt{{#1}}}
\newcommand{\iv}[1]{\lbrace\!|{#1}|\!\rbrace}

\newcommand{\epi}{\keyword{epi}}
\newcommand{\cl}{\keyword{cl}}
\newcommand{\id}{\keyword{id}}
\newcommand{\op}{\keyword{op}}
\newcommand{\conv}{\,\operatorname{\Box}\,}
\newcommand{\z}{\mathbf 0}
\newcommand{\N}{\mathcal N}
\newcommand{\im}{\keyword{im}}

\newcommand{\logpdf}{\keyword{logpdf}}
\newcommand{\cgf}{\keyword{cgf}}
\newcommand{\g}{-}

\newcommand{\cxto}{\rightharpoonup}
\newcommand{\cvto}{\rightharpoondown}
\newcommand{\exR}{\overline{\mathbb R}}

\newcommand{\white}{\circ}
\newcommand{\black}{\bullet}

\newcommand{\cxbifn}{\catname{CxBiFn}}
\newcommand{\cvbifn}{\catname{CvBiFn}}
\newcommand{\gauss}{\mathsf{Gauss}}

\newcommand{\cxcirc}{\stackrel{\cxto}{\circ}}
\newcommand{\cvcirc}{\stackrel{\cvto}{\circ}}

\newcommand{\graph}[1]{\underline{{#1}}}
\newcommand{\inff}[2]{\inf_{{#1}}\,\left\lbrace{#2}\right\rbrace}
\newcommand{\supp}[2]{\sup_{{#1}}\,\left\lbrace{#2}\right\rbrace}

\begin{document}

\title{Towards a Compositional Framework for Convex Analysis (with Applications to Probability Theory)}

\titlerunning{A Compositional Framework for Convex Analysis}

\author{Dario Stein\inst{1} \and Richard Samuelson\inst{2}}
\authorrunning{Stein and Samuelson}
\institute{Radboud University Nijmegen, The Netherlands, \email{dario.stein@ru.nl} \and
University of Florida, Gainesville, USA, \email{rsamuelson@ufl.edu}}

\maketitle

\begin{abstract}
We introduce a compositional framework for convex analysis based on the notion of \emph{convex bifunction} of Rockafellar. This framework is well-suited to graphical reasoning, and exhibits rich dualities such as the Legendre-Fenchel transform, while generalizing formalisms like graphical linear algebra, convex relations and convex programming. We connect our framework to probability theory by interpreting the Laplace approximation in its context: The exactness of this approximation on normal distributions means that logdensity is a functor from Gaussian probability (densities and integration) to concave bifunctions and maximization.
\keywords{convex analysis  \and category theory \and categorical probability}
\end{abstract}

\section{Introduction}

Convex analysis is a classical area of mathematics with innumerous applications in engineering, economics, physics, statistics and information theory. The central notion is that of a convex function $f : \R^n \to \R$, satisfying the inequality $f(tx + (1-t)y) \leq tf(x) + (1-t)f(y)$ for all $t \in [0,1]$. Convexity is a useful property for optimization problems: Every local minimum of $f$ is automatically a global minimum. Convex functions furthermore admit a beautiful duality theory; the ubiquitous Legendre-Fenchel transform (or convex conjugation) defined as
\[ f^*(x^*) = \supp x { \langle x^*,x\rangle - f(x) } \]
encodes $f$ in terms of all affine functions $\langle x^*,x\rangle - c$ majorized by $f$ (here $\langle -,-\rangle$ denotes the standard inner product on $\R^n$). The function $f^*$ is convex regardless of $f$, and under a closedness assumption we recover $f^{**}=f$.

While convex analysis is a rich field, its compositional structure is not readily apparent; the central notion, convex functions, is not closed under composition. The notion which \emph{does} compose is less well known: a \emph{convex bifunction}, due to \cite{rockafellar}, is a jointly convex function $F : \R^{m}\times\R^n \to \R$ of two variables. Such bifunctions compose via infimization
\[ (F \circ G)(x,z) = \inff y {F(y,z) + G(x,y)} \]

\paragraph{Categorical Methods} In this work, we will study bifunctions and their associated dualities in the framework of category theory. Graphical methods are ubiquitous in engineering and statistics, and can used to derive efficient algorithms by making use of the factorized structure of a problem. The language of props and string diagrams unifies these methods, as a large body of work on graphical linear algebra and applied category theory shows \cite{baez:control,baez:props,hanks2023compositional,bonchi:cat_signal_flow}. We extend these methods to problems of convex analysis and optimization. Our category of bifunctions subsumes an array of mathematical structures, such as linear maps and relations, convex relations, and (surprisingly) multivariate Gaussian probability. \vspace{-0.4cm}

\begin{figure}
    \centering
    \includegraphics[scale=0.38]{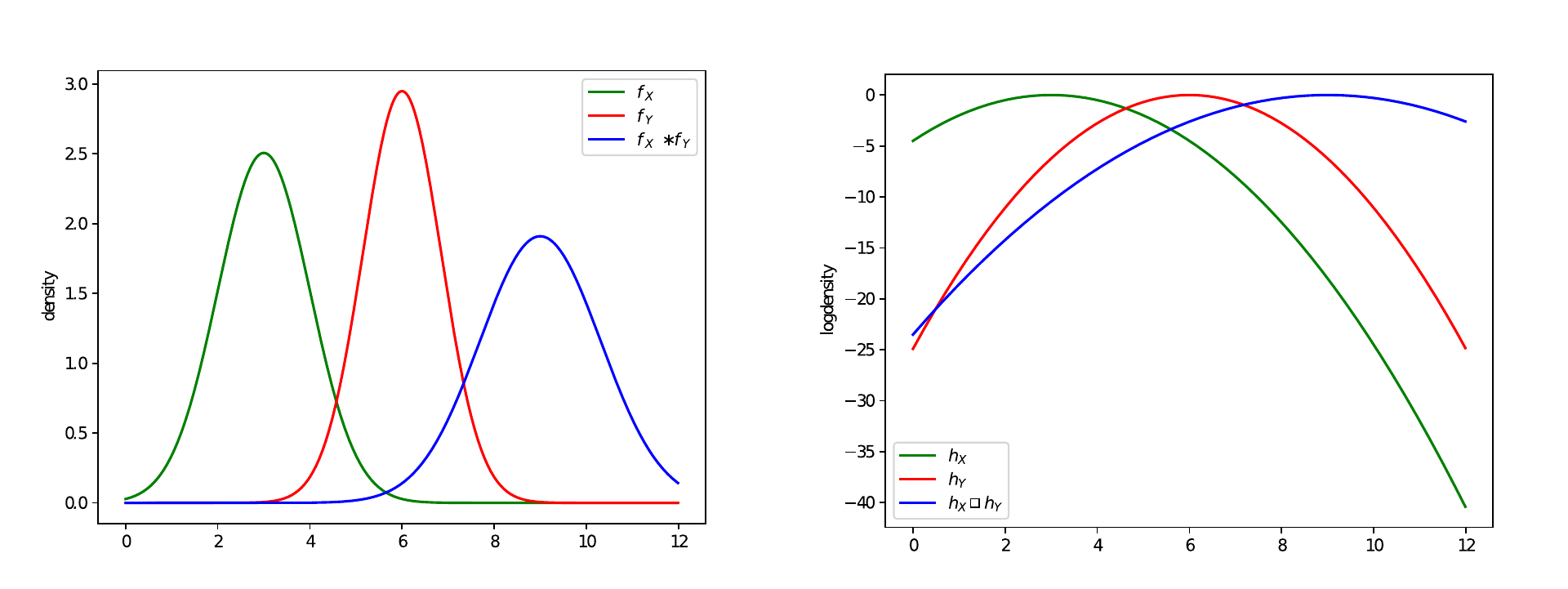}
    \caption{Addition of independent normal variables $X,Y$. Left: pdf and convolution, right: logpdf and sup-convolution}
    \label{fig:enter-label}
\end{figure}

\vspace{-0.5cm}
\paragraph{Applications to Probability Theory}

Convex analysis offers a rich perspective on Gaussian (multivariate normal) probability distributions: The log-density $h(x) = \log f(x)$ of a Gaussian random variable is a concave function of the form\footnote{we intentionally disregard a scalar $+C$} 
\[ h(x) = -\frac{(x-\mu)^2}{2\sigma^2} \]
It turns out that anything we can do with Gaussian densities and integration can instead be done with logdensities and maximization. For example, to compute the density of a sum of independent variables, we may take a convolution of densities, or instead compute a sup-convolution of logdensities (see Fig.~\ref{fig:enter-label}), as 
\[ \log \int f_X(x)f_Y(z-x) \mathrm d x = \supp x {h_X(x) + h_Y(z-x)} \]

This is highly particular to Gaussians. We can elegantly formalize this statement in categorical terms, as our main theorem states: Logprobability defines a functor from Gaussian probability to concave bifunctions (Theorem~\ref{thm:gauss_functor})

In this sense, the essence of Gaussians is captured by concave quadratic functions. By extending our viewpoint to partial concave quadratic functions, we obtain a generalized notion of \emph{Gaussian relation} which includes improper priors. Such entities are subtle to describe measure-theoretically, but straightforward in the convex analytic view. The duality theory of bifunctions generalizes the duality of precision and covariance, and more generally connects to the notion of cumulant-generating function in probability theory. 

We elegantly formalize the connections between convex analysis and probability theory using the language of Markov categories \cite{fritz2020synthetic}, which are a categorical formalism for probability theory, and have close connections to the semantics of probabilistic programs \cite{stein2021structural}.

\paragraph{Contribution and Outline}

This paper is intended to serve as a high-level roadmap to a categorical treatment of convex analysis. Our aim is to spell out the underlying structures, and present a diverse range of connections, especially with diagrammatic methods and categorical probability. For the sake of presentation, we choose to stick to general statements and keep some technical notions (such as regularity conditions) informal. Spelling out the details in a concrete setting is a starting point for future developments. We elaborate one such particular setting in detail, namely Gaussian probability. \\

We begin \Sec\ref{sec:intro_cx} by recalling the relevant notions of convex analysis, and proceed to define and study the categorical structure of bifunctions in \Sec\ref{sec:bifn}. This includes two structures as a hypergraph category and the duality theory of \Sec\ref{sec:bifn_duality}.

In \Sec\ref{sec:examples}, we elaborate different examples of categories which embed in bifunctions, such as linear and affine algebra, linear algebra, convex relations and convex optimization problems. In each case, the embedding preserves the relevant categorical structures and dualities. In particular, we show that the theory of bifunctions is a conservative extension of graphical linear algebra \cite{graphical_la}.

In \Sec\ref{sec:gaussians} we begin making connections to probability theory. We recall Gaussian probability from a categorical point of view, and construct the embedding functor to bifunctions. We discuss how partial quadratic functions can be seen as an extension of Gaussian probability beyond measure theory.

We conclude with \Sec\ref{sec:perspective}-\ref{sec:related} discussing the wider context of this work, elaborating connections of probability and convex analysis such as the Laplace approximation and cumulant generating functions, and the idea of idempotent analysis as a `tropical limit' of ordinary analysis.

\section{Overview of Convex Analysis}\label{sec:intro_cx}

The following section is a brief overview of standard material in convex analysis; all propositions and conventions are taken from \cite{rockafellar}.

\emph{Caveat:} An important feature of convex analysis is that it deals with formal infinities $+\infty, -\infty$ in a consistent fashion. This is crucial because optimization problems may be unbounded. Traditionally, one considers the extended real numbers $\exR = [-\infty, +\infty]$ and extends the usual laws of arithmetic to them. The case $(+\infty) + (-\infty)$ is left undefined and carefully avoided like $0/0$ in real analysis. A more systematic approach \cite{willerton,fujii2019} is based on enriched category theory, and endows $\exR$ with the structure of a commutative quantale, which gives it totally defined operations with a particular arithmetic.

\emph{A more serious caveat} is that many results in convex analysis require specific regularity assumptions to hold. As these assumptions are not the focus of the present paper, so we will state some big picture theorems in \Sec\ref{sec:bifn} under reservation of these conditions. We then elaborate an array of concrete examples \Sec\ref{sec:examples}-\ref{sec:gaussians} where we make sure that all regularity conditions are indeed satisfied. We discuss this drawback in \Sec\ref{sec:related}. \\

A subset $A \subseteq \R^n$ is \emph{convex} if for all $x,y \in A$ and $t \in [0,1]$, we have $tx + (1-t)y \in A$. The \emph{epigraph} of a function $f : \R^n \to \exR$ is the set $\epi(f) = \{ (x,y) \in \R^{n+1} : y \geq f(x) \}$. We say that $f$ is convex if $\epi(f)$ is a convex subset of $\R^{n+1}$. This is equivalent to the well-known definition from the introduction, while accounting for infinities. We say that $f$ is \emph{concave} if $(-f)$ is convex. 

\begin{example}
The following functions are convex: linear functions, $|x|$, $x^2$, $\exp(x)$, $-\ln(x)$. 
For a convex subset $A \subseteq \R^n$, the \emph{convex indicator function} $\delta_A : \R^n \to \exR$ is defined by 
\[ \delta_A(x) = \begin{cases}
0, &x \in A \\
+\infty, &x \notin A
\end{cases}
\]
We also write indicator functions using modified Iverson brackets as $\iv{x \in A} = \delta_A(x)$. The \emph{concave indicator function} of $A$ is $-\delta_A(x)$. 
\end{example} \vspace{-0.3cm}
The \emph{infimal convolution} of convex functions $f,g : \R^n \to \exR$ is defined by $(f \conv g)(x) = \inff y {f(x-y) + g(y)}$. The convex function $f$ is called \emph{closed} if $\epi(f)$ is a closed subset of $\R^{n+1}$; this is equivalent to $f$ being lower semicontinuous.

\subsection{Conjugacy -- the Legendre-Fenchel transform}

\begin{definition}
For a convex function $f : \R^n \to \exR$, its convex conjugate $f^* : \R^n \to \exR$ is the convex function  
\[ f^*(x^*) = \supp x {\langle x^*,x\rangle - f(x)} \]
For a concave function $g : \R^n \to \exR$, its concave conjugate $g^* : \R^n \to \exR$ is the concave function 
\[ g^*(x^*) = \inff x {\langle x^*,x \rangle - g(x) } \]
Note that if $g=-f$ then $g^*(x^*)=-f^*(-x^*)$
\end{definition}
Geometrically, $f^*$ encodes information about which affine functions $\langle x^*,-\rangle - c$ are majorized by $f$. It is thus natural to view $f^*$ as a function on covectors $x^* \in (\R^n)^*$. This is for example done in \cite{willerton}, but in order to keep notation consistent with \cite{rockafellar}, we make the traditional identification $(\R^n)^* \cong \R^n$ via the inner product, and the notation $x^*$ is purely decoration. The Legendre-Fenchel transform has applications in many areas of mathematics and physics \cite{touchette2005legendre}, such as the Hamiltonian formalism in mechanics, statistical mechanics or large deviation theory (e.g. \Sec\ref{sec:cgf}). \\

A closed convex function $f$ is the pointwise supremum of all affine functions $h \leq f$ \cite[12.1]{rockafellar}. This allows them to be recovered by their Legendre transform

\begin{proposition}[{\cite[Theorem~12.2]{rockafellar}}]\label{prop:double_conjugate}
For any convex function $f : \R^n \to \exR$, $f^*$ is a closed convex function. We have $f^{**} = f$ if and only if $f$ is closed. 
\end{proposition}
For arbitrary functions $f$, the operation $f \mapsto f^{**}$ is a closure operator which we denote by $\cl(f)$. This is the largest closed convex function majorized by $f$. 

\begin{example}\label{prop:ex_abs}
The absolute value function $f(x) = |x|$ is convex and closed. The supremum $\supp {x} {cx-|x|}$ equals $0$ if $|c| \leq 1$, and $\infty$ otherwise. Hence $f^*(c) = \iv{|c| \leq 1}$, and $f^{**}=f$.
\end{example}

\begin{example}
Let $f(x) = ax^2$ for $a > 0$. Then $x \mapsto c \cdot x - a x^2$ is differentiable and has a maximum at $x=c/2a$. We obtain $f^*(c) = \frac 1 {4a} c^2$. In particular, we see that the function $f(x) = \frac 1 2 x^2$ is a fixed point of the Legendre transform.
\end{example}

\begin{proposition}[{\cite[Theorem~16.4]{rockafellar}}]
If $f, g$ are closed convex functions, then under certain regularity conditions $(f \conv g)^* = f^* + g^*$ and $(f + g)^* = f^* \conv g^*$.
\end{proposition}

\section{Categories of Convex Bifunctions}\label{sec:bifn}

We now come to the central definition of this article, namely that of convex (or concave) \emph{bifunctions}. This concept is due to \cite{rockafellar} and scattered throughout his book. 

A \emph{bifunction} $F$ from $\R^m$ to $\R^n$ is the convex analysis terminology for a curried function $\R^m \to (\R^n \to \exR)$. The uncurried function $\graph{F} : \R^{m+n} \to \exR$ is referred to as the \emph{graph function} of $F$. We will suppress the partial application and write $F(x)(y)$ and $F(x,y)$ interchangeably.

\begin{definition}
A bifunction $F$ from $\R^m$ to $\R^n$ is called \emph{convex} (or \emph{concave}, \emph{closed}) if its graph function $\graph{F} : \R^{m + n} \to \exR$ has that property. The closure operation $cl(F)$ is applied on the level of graph functions. We denote a convex bifunction by $F : \R^m \cxto \R^n$ and a concave bifunction by $F : \R^m \cvto \R^n$.
\end{definition}

Bifunction composition is known as \emph{product} in \cite[\S~38]{rockafellar}.

\begin{definition}[Categories of bifunctions]
We define a category $\cxbifn$ of convex bifunctions as follows
\begin{itemize}
\item objects are the spaces $\R^n$
\item morphisms are convex bifunctions $\R^m \cxto \R^n$ 
\item the identity $\R^n \cxto \R^n$ is given by the indicator function
\[ \id_n(x,y) =\iv{x=y} \]
\item composition is infimization over the middle variable
\[ (F \cxcirc G)(x,z) = \inff y {G(x,y) + F(y,z)} \]
\end{itemize}
Analogously, the category $\cvbifn$ of concave bifunctions is defined as
\begin{itemize}
\item objects are the spaces $\R^n$
\item morphisms are concave bifunctions $\R^m \cvto \R^n$ 
\item the identity $\R^n \cvto \R^n$ is given by the concave indicator function
\[ -\id_n(x,y) = -\iv{x=y} \]
\item composition is supremization over the middle variable
\[ (F \cvcirc G)(x,z) = \supp y {G(x,y) + F(y,z)} \]
\end{itemize}
\end{definition}
\begin{proof}[of well-definedness]
This construction is a subcategory of the the category of weighted relations $\catname{Rel}(Q)$ taking values in a commutative quantale $Q$ \cite{bolt2019,coecke2018,marsden2017}, where $Q=\exR$ are the extended reals. It suffices to verify that convex bifunctions are closed under composition, tensor (addition) and contain the identities (\cite[p.~408]{rockafellar}).
\end{proof}

We will write bifunction composition as $F \circ G$ when it is clear from context whether we use the convex or concave variety. We will write $I$ for the unit space $\R^0$, and $\z$ for its unique element. 

\begin{example}
The \emph{states} (morphisms $I \cxto \R^n$ out of the unit) are in bijection with convex functions $f : \R^n \to \exR$, as are the \emph{effects} $\R^n \cxto I$. States and effects in $\cvbifn$ are in bijection with concave functions $f : \R^n \to \exR$.
\end{example}

\subsection{Duality for Bifunctions}\label{sec:bifn_duality}

Unless otherwise stated, theorems phrased for convex bifunctions will hold for concave bifunctions by selecting the appropriate versions of the operations.

The duality theory of convex functions extends to bifunctions as follows. 

\begin{definition}[{\cite[\S 30]{rockafellar}}]
The \emph{adjoint} of a convex bifunction $F : \R^m \cxto \R^n$ is the \emph{concave} bifunction $F^* : \R^n \cvto \R^m$ defined by
\[ F^*(y^*,x^*) = \inff {x,y} {F(x,y) + \langle x^*,x\rangle - \langle y^*,y\rangle } \]
The adjoint of a concave bifunction is convex and uses $\sup$ instead of $\inf$. The adjoint of the convex bifunction $F$ is related to the conjugate of its graph function $\graph{F}$ using the formula $F^*(y^*,x^*) = -\graph{F}^*(-x^*,y^*)$. (Note the slight asymmetry that one input is negated)
\end{definition}

The analogue of Proposition~\ref{prop:double_conjugate} for bifunctions is as follows

\begin{proposition}[{\cite[Theorem~30.1]{rockafellar}}]
For any convex bifunction $F$, the adjoint $F^*$ is a closed concave bifunction, and we have $F^{**} = \cl(F)$. In particular, if $F$ is a closed convex bifunction, then $F^{**} = F$.
\end{proposition}

\begin{theorem}[{\cite[Theorem~38.5]{rockafellar}}]\label{thm:adjoint_functoriality}
Under regularity assumptions, the adjoint operation respects composition. That is, for $F : \R^m \cxto \R^n$ and $G : \R^n \cxto \R^k$, we have
\[ (G \cxcirc F)^* = F^* \cvcirc G^* \]
That is, the adjoint operation defines a pair of mutually inverse functors
\[\begin{tikzcd}[ampersand replacement=\&]
	{\cxbifn^\op} \&\& \cvbifn
	\arrow["{(-)^*}", curve={height=-6pt}, dashed, from=1-1, to=1-3]
	\arrow["{(-)^*}", curve={height=-6pt}, dashed, from=1-3, to=1-1]
\end{tikzcd}\]
We indicate with dashed arrows that the functoriality depends on regularity assumptions.
\end{theorem}

For the interested reader, the regularity assumptions in Theorem~\ref{thm:adjoint_functoriality} include closedness, as well as properness and certain (relative interiors of) domains of the involved bifunctions intersecting \cite[\S~38]{rockafellar}. These assumptions are not necessary conditions.

As a corollary of functoriality, we can derive the following well-known fact:
\begin{corollary}[Fenchel duality]
Let $f : \R^n \to \exR$ be convex, $g : \R^n \to \exR$ concave, and let $f^*,g^*$ be their convex and concave conjugates respectively. Then under sufficient regularity assumptions, we have
\[ \inff x {f(x) - g(x)} = \supp {x^*} {g^*(x^*) - f^*(x^*)} \]
\end{corollary}
\begin{proof}
Consider the convex function $h=-g$ and form the state $s_f : I \cxto \R^n, s_f(\z,x) = f(x)$ and effect $e_h : \R^n \cxto I, e_h(x,\z) = h(x)$. The proof proceeds by using functoriality to compute the scalar $(e_h \cxcirc s_f)^* = (s_f^* \cvcirc e_h^*)$ in two ways: On the one hand, we have
\[ (e_h \cxcirc s_f)(\z,\z) = \inff x {s_f(\z,x) + e_h(x,\z)} = \inff x {f(x)-g(x)} \]
On the other hand, we express the adjoints in terms of the conjugates $f^*,g^*$
\begin{align*}
s_f^*(x^*,\z) &= \inff x {s_f(\z,x) - \langle x^*,x \rangle} = -f^*(x^*) \\
e_h^*(\z,x^*) &= \inff x {e_h(x,\z) + \langle x^*,x \rangle} = g^*(x^*)
\end{align*}
The adjoint acts as the identity on scalars, so we obtain
\[ \inff x {f(x)-g(x)} = (e_h \cxcirc s_f)^*(\z,\z) = (s_f^* \cvcirc e_h^*)(\z,\z) = \supp {x} {g^*(x^*) - f^*(x^*)}\]
\end{proof}

\subsection{Hypergraph Structure and Symmetries}

Bifunctions can not only be composed in sequence, but also in parallel. The relevant structure is that of a symmetric monoidal category $(\mathbb C, \otimes, I)$. In this work, we are dealing with a particular simple form of such categories called a \emph{prop}. A prop $\mathbb C$ is a strict monoidal category which is generated by a single object $R$ so that every object is of the form $R^{\otimes n}$ for some $n \in \mathbb N$. The monoid of objects $(\keyword{ob}(\mathbb C), \otimes, I)$ is thus isomorphic to $(\mathbb N,+,0)$.

\begin{proposition}
Convex bifunctions have the structure of a prop, generated by the object $\R$
\begin{enumerate}
\item The tensor is $\R^m \otimes \R^n = \R^{m + n}$
\item The unit is $I = \R^0$. 
\item The tensor of bifunctions is given by addition: If $F : \R^{m_1} \cxto \R^{n_1}$, $G : \R^{m_2} \cxto \R^{n_2}$ then $F \otimes G : \R^{m_1+m_2} \cxto \R^{n_1 + n_2}$ is defined as
\[ (F \otimes G)((x_1,x_2),(y_1,y_2)) = F(x_1,y_1) + G(x_2,y_2) \]
\end{enumerate}
\end{proposition}
\begin{proof}[of well-definedness]
General fact about categories of weighted relations $\catname{Rel}(Q)$ (\cite{marsden2017}).
\end{proof}

Symmetric monoidal categories are widely studied and admit a convenient graphical language using string diagrams \cite{selinger:graphical}. It is useful to consider further pieces of structure on such a category
\begin{enumerate}
\item in a \emph{copy-delete category} \cite{cho_jacobs}, every object carries the structure of a commutative comonoid $\keyword{copy}_X : X \to X \otimes X$ and $\keyword{discard}_X : X \to I$. This lets information be used in a non-linear way (in the sense of linear logic).
\item in a \emph{hypergraph category} \cite{fong2019hypergraph}, every object carries the structure of a special commutative Frobenius algebra
\end{enumerate}
Every hypergraph category is in particular a copy-delete category. The pieces of structure of a hypergraph category are often rendered as cups and caps in string diagrams

\[ \begin{tikzpicture}[scale=0.4]
	\begin{pgfonlayer}{nodelayer}
		\node [style=none] (0) at (12, 2) {};
		\node [style=bn] (1) at (10, 2) {};
		\node [style=none] (2) at (-10, 2) {};
		\node [style=bn] (3) at (-8, 2) {};
		\node [style=none] (4) at (-6, 3) {};
		\node [style=none] (5) at (-6, 1) {};
		\node [style=none] (6) at (3, 3) {};
		\node [style=none] (7) at (3, 1) {};
		\node [style=bn] (8) at (5, 2) {};
		\node [style=none] (9) at (7, 2) {};
		\node [style=none] (10) at (-3.25, 2) {};
		\node [style=bn] (11) at (-1.25, 2) {};
		\node [style=none] (12) at (-8, -1) {copy};
		\node [style=none] (13) at (-2.5, -1) {discard};
		\node [style=none] (14) at (5, -1) {multiply};
		\node [style=none] (15) at (11, -1) {unit};
	\end{pgfonlayer}
	\begin{pgfonlayer}{edgelayer}
		\draw (0.center) to (1);
		\draw (2.center) to (3);
		\draw [in=180, out=90] (3) to (4.center);
		\draw [in=-180, out=-90] (3) to (5.center);
		\draw [in=-90, out=0] (7.center) to (8);
		\draw [in=0, out=90] (8) to (6.center);
		\draw (8) to (9.center);
		\draw (11) to (10.center);
	\end{pgfonlayer}
\end{tikzpicture}
 \]

subject to equations such as the Frobenius law

\[ \begin{tikzpicture}[scale=0.3]
	\begin{pgfonlayer}{nodelayer}
		\node [style=none] (2) at (-12, -1) {};
		\node [style=bn] (3) at (-10, -1) {};
		\node [style=none] (4) at (-8, 0) {};
		\node [style=none] (5) at (-8, -2) {};
		\node [style=none] (6) at (-8, 2) {};
		\node [style=none] (7) at (-8, 0) {};
		\node [style=bn] (8) at (-6, 1) {};
		\node [style=none] (9) at (-4, 1) {};
		\node [style=none] (10) at (-12, 2) {};
		\node [style=none] (11) at (-4, -2) {};
		\node [style=none] (12) at (4.25, 0) {};
		\node [style=bn] (13) at (6.25, 0) {};
		\node [style=none] (14) at (8.25, 1) {};
		\node [style=none] (15) at (8.25, -1) {};
		\node [style=none] (16) at (8.25, 1) {};
		\node [style=none] (17) at (0.25, -1) {};
		\node [style=none] (18) at (0.25, 1) {};
		\node [style=none] (19) at (0.25, -1) {};
		\node [style=bn] (20) at (2.25, 0) {};
		\node [style=none] (21) at (4.25, 0) {};
		\node [style=none] (22) at (12, 1) {};
		\node [style=bn] (23) at (14, 1) {};
		\node [style=none] (24) at (16, 2) {};
		\node [style=none] (25) at (16, 0) {};
		\node [style=none] (26) at (16, 2) {};
		\node [style=none] (27) at (16, -2) {};
		\node [style=none] (28) at (16, 0) {};
		\node [style=none] (29) at (16, -2) {};
		\node [style=bn] (30) at (18, -1) {};
		\node [style=none] (31) at (20, -1) {};
		\node [style=none] (32) at (12, -2) {};
		\node [style=none] (33) at (20, 2) {};
		\node [style=none] (34) at (-2, 0) {=};
		\node [style=none] (35) at (10, 0) {=};
	\end{pgfonlayer}
	\begin{pgfonlayer}{edgelayer}
		\draw (2.center) to (3);
		\draw [in=180, out=90] (3) to (4.center);
		\draw [in=-180, out=-90] (3) to (5.center);
		\draw [in=-90, out=0] (7.center) to (8);
		\draw [in=0, out=90] (8) to (6.center);
		\draw (8) to (9.center);
		\draw (10.center) to (6.center);
		\draw (11.center) to (5.center);
		\draw (12.center) to (13);
		\draw [in=180, out=90] (13) to (14.center);
		\draw [in=-180, out=-90] (13) to (15.center);
		\draw [in=-90, out=0] (19.center) to (20);
		\draw [in=0, out=90] (20) to (18.center);
		\draw (20) to (21.center);
		\draw (22.center) to (23);
		\draw [in=180, out=90] (23) to (24.center);
		\draw [in=-180, out=-90] (23) to (25.center);
		\draw [in=-90, out=0] (29.center) to (30);
		\draw [in=0, out=90] (30) to (28.center);
		\draw (30) to (31.center);
		\draw (32.center) to (29.center);
		\draw (33.center) to (26.center);
	\end{pgfonlayer}
\end{tikzpicture}
 \]

This gives rise to a rich graphical calculus, which has been explored for a number of engineering applications like signal-flow diagrams or electrical circuits \cite{graphical_la,bsz,bonchi:cat_signal_flow,bonchi:interacting_hopf,baez:control,baez:props}

\begin{proposition}
$\cxbifn$ has the structure of a hypergraph category in two different ways, which we call the additive and co-additive structure. That is, every object carries two different structures as a special commutative Frobenius algebra 
\begin{enumerate}
\item The additive structure is given by
\begin{align*}
\keyword{unit} : I \cxto \R^n,  &\qquad  \keyword{unit}(\z,x) = 0 \\
\keyword{discard} : \R^n \cxto I, &\qquad  \keyword{discard}(x,\z) = 0 \\
\keyword{copy} : \R^n \cxto \R^n \otimes \R^n &\qquad  \keyword{copy}(x,y,z) = \iv{x=y=z} \\
\keyword{comp} : \R^n \otimes \R^n \cxto \R^n, &\qquad  \keyword{comp}(x,y,z) = \iv{x=y=z} 
\end{align*}
\item The co-additive structure is given by
\begin{align*}
\keyword{zero} : I \cxto \R^n,  &\qquad  \keyword{zero}(\z,x) = \iv{x=0} \\
\keyword{cozero} : \R^n \cxto I, &\qquad  \keyword{cozero}(x,\z) = \iv{x=0} \\
\keyword{add} : \R^n \otimes \R^n \cxto \R^n, &\qquad  \keyword{add}(x,y,z) = \iv{x+y=z} \\
\keyword{coadd} : \R^n \cxto \R^n \otimes \R^n, &\qquad  \keyword{coadd}(z,x,y) = \iv{x+y=z} 
\end{align*}
\end{enumerate}
The analogous structures on $\cvbifn$ use concave indicator functions instead. 
\end{proposition}

We can motivate the names of the hypergraph structures by observing how multiplications acts on states. This duality is clarified in what follows.
\begin{example}
Let $f,g : I \cxto \R^n$ be two states. Then
\begin{align*}
(\keyword{copy} \circ (f \otimes g))(z) &= \inff {x,y} {f(x) + g(y) + \iv{x=y=z}} = f(z) + g(z) \\
(\keyword{add} \circ (f \otimes g))(z) &= \inff {x,y} {f(x) + g(y) + \iv{x+y=z}} = f(z) \conv g(z) \\
\end{align*}
\end{example}

\begin{definition}
The \emph{dagger} of a bifunction $F : \R^m \cxto \R^n$ is given by reversing its arguments 
\[ F^\dagger : \R^n \cxto \R^m, F^\dagger(y,x) = F(x,y) \]
The \emph{inverse} of a bifunction $F : \R^m \cxto \R^n$ is the concave bifunction \cite[p.~384]{rockafellar}
\[ F_*(x,y) = -F(y,x) \]
Both these operations define involutive\footnote{i.e. applying the appropriate version of these operations twice is the identity} functors
\begin{align*}
(-)^\dagger : \cxbifn^\op \to \cxbifn, \qquad (-)_*  : \cxbifn^\op \to \cvbifn
\end{align*}
The functor $(-)^\dagger$ is a dagger functor in the sense of \cite{selinger2007dagger}.
\end{definition}

\begin{proposition}[{\cite[p.~384]{rockafellar}}]\label{prop:inv_adj_commute}
The operations of inverse and adjoint commute, i.e. for $F : \R^m \cxto \R^n$ we have $(F^*)_* = (F_*)^*$.
\end{proposition}
% This elegant property underlines why it is convenient to consider the pair of operations $(-)^*$ and $(-)_*$ which interchange convex and concave. If one instead defined an operation $F^? = -F^*$ to take convex to convex bifunctions, then the pair of operations $(-)^?$ and $(-)^\dagger$ does not commute. \\

The composite operation $F^*_*$ defines another covariant functor $\cxbifn \to \cxbifn$, which we now interpret: As is customary in graphical linear algebra, we render the two hypergraph structures as follows
\begin{equation} \begin{tikzpicture}[scale=0.4]
	\begin{pgfonlayer}{nodelayer}
		\node [style=none] (0) at (3, 2) {};
		\node [style=bn] (1) at (1, 2) {};
		\node [style=none] (2) at (-10, 2) {};
		\node [style=bn] (3) at (-8, 2) {};
		\node [style=none] (4) at (-6, 3) {};
		\node [style=none] (5) at (-6, 1) {};
		\node [style=none] (6) at (6, 3) {};
		\node [style=none] (7) at (6, 1) {};
		\node [style=bn] (8) at (8, 2) {};
		\node [style=none] (9) at (10, 2) {};
		\node [style=none] (10) at (-2.75, 2) {};
		\node [style=bn] (11) at (-1, 2) {};
		\node [style=none] (12) at (-8, -1) {copy};
		\node [style=none] (13) at (-2, -1) {discard};
		\node [style=none] (14) at (8, -1) {comp};
		\node [style=none] (15) at (2, -1) {unit};
		\node [style=none] (16) at (3, -5) {};
		\node [style=wn] (17) at (1, -5) {};
		\node [style=none] (18) at (-10, -5) {};
		\node [style=wn] (19) at (-8, -5) {};
		\node [style=none] (20) at (-6, -4) {};
		\node [style=none] (21) at (-6, -6) {};
		\node [style=none] (22) at (6, -4) {};
		\node [style=none] (23) at (6, -6) {};
		\node [style=wn] (24) at (8, -5) {};
		\node [style=none] (25) at (10, -5) {};
		\node [style=none] (26) at (-3, -5) {};
		\node [style=wn] (27) at (-1, -5) {};
		\node [style=none] (28) at (-8, -8) {coadd};
		\node [style=none] (29) at (-2, -8) {cozero};
		\node [style=none] (30) at (8, -8) {add};
		\node [style=none] (31) at (2, -8) {zero};
	\end{pgfonlayer}
	\begin{pgfonlayer}{edgelayer}
		\draw (0.center) to (1);
		\draw (2.center) to (3);
		\draw [in=180, out=90] (3) to (4.center);
		\draw [in=-180, out=-90] (3) to (5.center);
		\draw [in=-90, out=0] (7.center) to (8);
		\draw [in=0, out=90] (8) to (6.center);
		\draw (8) to (9.center);
		\draw (11) to (10.center);
		\draw (16.center) to (17);
		\draw (18.center) to (19);
		\draw [in=180, out=90] (19) to (20.center);
		\draw [in=-180, out=-90] (19) to (21.center);
		\draw [in=-90, out=0] (23.center) to (24);
		\draw [in=0, out=90] (24) to (22.center);
		\draw (24) to (25.center);
		\draw (27) to (26.center);
	\end{pgfonlayer}
\end{tikzpicture}
 \label{eq:gla_generators} \end{equation}
We refer to the additive structure as `black' ($\black$) and the co-additive one as `white' ($\white$). This presentation reveals an array of symmetries (mirror-image and color-swap\footnote{we will discuss these symmetries in more detail in Section~\ref{sec:linearalgebra}}), which we are relating now:

\begin{theorem}
The adjoint operation interchanges the additive and co-additive structure. That is we have functors of hypergraph categories
\begin{align*}
(-)^* : (\cxbifn^\op, \black) \to (\cvbifn, \white)  \\
(-)^* : (\cxbifn^\op, \white) \to (\cvbifn, \black)  
\end{align*}
Note that the opposite of a hypergraph category is again a hypergraph category where cups and caps are interchanged.
\end{theorem}
\begin{proof}
Follows from the results in \Sec\ref{sec:linearalgebra}, as the hypergraph structures are induced by linear maps.
\end{proof}
In terms of the generators \eqref{eq:gla_generators}, the mirror image is given by the $(-)^\dagger$ functor. Both hypergraph structures consist of $\dagger$-Frobenius algebras, meaning that $(-)^\dagger$ is a functor of hypergraph categories $\cxbifn^\op \to \cxbifn$.

The color-swap operation is given by the inverse adjoint $F^*_*$, which gives a hypergraph equivalence $(\cxbifn,\black) \to (\cxbifn,\white)$. This equivalence does however not commute with $\dagger$, i.e. is not an equivalence of dagger hypergraph categories.

\section{Example Categories of Bifunctions}\label{sec:examples}

We now elaborate example subcategories of bifunctions on which functoriality and duality work as desired (that is, all regularity conditions apply).  

\subsection{Linear Algebra}\label{sec:linearalgebra}

The identities and dualities of convex bifunctions generalize those of linear algebra. Let $A : \R^m \to \R^n$ be a linear map. The convex indicator bifunction of $A$ is defined as
\[ F_A(x,y) = \iv{y = Ax} \]
The following facts hold \cite[p~310]{rockafellar}:
\begin{enumerate}
\item For composable linear maps, $A,B$ we have $F_{AB} = F_A \circ F_B$
\item The adjoint $F_A^*$ is the concave indicator bifunction of the transpose $A^T$ 
\[ F_A^*(y^*,x^*) = -\iv{x^* = A^Ty^*} \]
\item if $A$ is invertible, then the inverse $(F_A)_*$ is the concave indicator bifunction associated to the inverse $A^{-1}$. In that case, Proposition~\ref{prop:inv_adj_commute} generalizes the identity $(A^{-1})^T = (A^T)^{-1}$.
\end{enumerate}

In more categorical terms, let $\catname{Vect}$ denote the prop of the vector spaces $\R^n$ and linear maps. This is a copy-delete category equipped with the linear maps $\Delta : \R^n \to \R^n \oplus \R^n$ and $! : \R^n \to \R^0$. For a linear map $A : \R^m \to \R^n$, define 
\begin{align*}
F_A : \R^m \cxto \R^n, F_A(x,y) = \iv{y=Ax} \\
G_A : \R^n \cvto \R^m, G_A(y,x) = -\iv{x=A^Ty}
\end{align*}
\begin{theorem}
We have a commutative diagram of functors between copy-delete categories
% https://q.uiver.app/#q=WzAsMyxbMSwwLCJcXGNhdG5hbWV7VmVjdH0iXSxbMiwyLCIoXFxjeGJpZm4sXFxibGFjaykiXSxbMCwyLCIoXFxjdmJpZm5ee1xcb3B9LFxcd2hpdGUpIl0sWzAsMSwiRiJdLFswLDIsIkciLDJdLFsyLDEsIigtKV4qIiwyLHsiY3VydmUiOi0xLCJzdHlsZSI6eyJib2R5Ijp7Im5hbWUiOiJkYXNoZWQifX19XSxbMSwyLCIiLDEseyJjdXJ2ZSI6LTEsInN0eWxlIjp7ImJvZHkiOnsibmFtZSI6ImRhc2hlZCJ9fX1dXQ==
\begin{equation}\begin{tikzcd}[ampersand replacement=\&]
	\& {\catname{Vect}} \\
	\\
	{(\cvbifn^{\op},\white)} \&\& {(\cxbifn,\black)}
	\arrow["F", from=1-2, to=3-3]
	\arrow["G"', from=1-2, to=3-1]
	\arrow["{(-)^*}"', curve={height=-6pt}, dashed, from=3-1, to=3-3]
	\arrow[curve={height=-6pt}, dashed, from=3-3, to=3-1]
\end{tikzcd}\label{eq:duality_vect}
\end{equation}
\end{theorem}
\begin{proof}
Functoriality and commutativity follow from the above facts. For the copy-delete structures, notice that $\keyword{copy}, \keyword{delete}, \keyword{add}, \keyword{zero}$ are the indicator bifunctions of the linear maps $\Delta$ and $!$. The transpose of $\Delta$ is summation $(x,y) \mapsto x+y$.
\end{proof}

We call a diagram like \eqref{eq:duality_vect} a \emph{duality situation}. The dashed arrows indicate that, while $(-)^*$ is neither a functor nor idempotent on all bifunctions without further conditions, everything works out on the image of $F,G$ respectively. We could thus obtain a genuine commutative diagram of functors by characterizing these images exactly (which we refrain from doing here for the sake of simplicity).

\paragraph{Linear Relations}

Graphical Linear Algebra \cite{graphical_la} is the diagrammatic study of the prop $\catname{LinRel}$ of linear relations, which are relations $R \subseteq \R^m \times \R^n$ that are also vector subspaces. This category is a hypergraph category using the two structures shown in \eqref{eq:gla_generators}, and the operations mirror-image and color-swap are defined for linear relations via relational converse and a twisted orthogonal complement
\begin{align*}
R^\dagger &= \{ (y,x) : (x,y) \in R \} \\
R^c &= \{ (x^*,y^*) : \forall (x,y) \in R, \langle x^*,x \rangle - \langle y^*,y \rangle = 0 \}
\end{align*}
The operations $(-)^\dagger$ and $(-)^c$ commute and define a composite contravariant involution $(-)^* : \catname{LinRel}^\op \to \catname{LinRel}$. The following theorem shows that bifunctions are a conservative extension of graphical linear algebra.

\begin{theorem}\label{thm:linrel_functor}
If we embed a linear relation $R \subseteq \R^m \times \R^n$ via its indicator function as a bifunction $I_R : \R^m \cxto \R^n$, then we have a commutative diagram
% https://q.uiver.app/#q=WzAsNCxbMCwwLCJcXGNhdG5hbWV7TGluUmVsfV5cXG9wIl0sWzIsMCwiXFxjYXRuYW1le0xpblJlbH0iXSxbMCwyLCJcXGN2Ymlmbl5cXG9wIl0sWzIsMiwiXFxjeGJpZm4iXSxbMCwxLCIiLDAseyJjdXJ2ZSI6LTF9XSxbMSwwLCIoLSleKiIsMix7ImN1cnZlIjotMX1dLFsyLDMsIiIsMix7ImN1cnZlIjotMSwic3R5bGUiOnsiYm9keSI6eyJuYW1lIjoiZGFzaGVkIn19fV0sWzMsMiwiKC0pXioiLDIseyJjdXJ2ZSI6LTEsInN0eWxlIjp7ImJvZHkiOnsibmFtZSI6ImRhc2hlZCJ9fX1dLFsxLDMsIkkiXSxbMCwyLCItSSIsMl1d
\begin{equation}\begin{tikzcd}[ampersand replacement=\&]
	{\catname{LinRel}^\op} \&\& {\catname{LinRel}} \\
	\\
	{\cvbifn^\op} \&\& \cxbifn
	\arrow[curve={height=-6pt}, from=1-1, to=1-3]
	\arrow["{(-)^*}"', curve={height=-6pt}, from=1-3, to=1-1]
	\arrow[curve={height=-6pt}, dashed, from=3-1, to=3-3]
	\arrow["{(-)^*}"', curve={height=-6pt}, dashed, from=3-3, to=3-1]
	\arrow["I", from=1-3, to=3-3]
	\arrow["{-I}"', from=1-1, to=3-1]
\end{tikzcd}\label{eq:duality_linrel}
\end{equation}
In addition, the functor $I$ preserves both hypergraph structures.
\end{theorem}

\paragraph{Affine Relations} Graphical linear algebra has been extended to affine relations \cite{pbsz}; those are affine subspaces $R \subseteq \R^m \times \R^n$. This still forms a hypergraph category with both structures $\black,\white$, however the color-swap symmetry of linear relations is broken. That is because the affine generator $\underline 1 : 0 \to 1$ representing the affine relation $\{(\z,1)\}$ does not have an obvious color-swapped dual; affine subspaces are not recovered by their orthogonal complements.

The embedding into bifunctions suggests an avenue to recover such a symmetry: Taking the embedding \eqref{eq:duality_linrel} as a starting point, the indicator bifunction of $\underline 1$ is $f : I \cxto \R, f(\z,x) = \iv{x=1}$. Its adjoint is $f^*(x^*,\z) = -x^*$, which is a perfectly well-defined bifunction but not the indicator bifunction of any affine relation. This suggests that an extension of affine relations with color-swap symmetry can be obtained using a category of partial affine function (e.g. \cite[p.~107]{rockafellar}) but details are to left for future work. We will discuss the case of partial quadratic functions in \Sec\ref{sec:pqf}.

\subsection{Convex Relations}

Generalizing the previous example even further, a \emph{convex relation} $R \subseteq \R^m \times \R^n$ is a relation which is also a convex subset of $\R^{m + n}$. Convex relations are closed under the usual relation composition and thus form a prop $\catname{CxRel}$ \cite{bolt2019,coecke2018,marsden2017}.

Every linear relation is in particular convex, and like linear relations, convex relations embed into convex bifunctions via the indicator function.

We sketch a certain converse to this embedding: The space $(\R,+,0)$ is a monoid object in $\catname{CxRel}$. We consider the `writer' monad $T : \catname{CxRel} \to \catname{CxRel}$ associated to that monoid, i.e. $T(\R^m) = \R^{m+1}$. If $S \subseteq \R^m \times \R^{n+1}$ and $R \subseteq \R^n \times \R^{k+1}$ are Kleisli arrows, then Kleisli composition takes the following form
\[ R \bullet S = \{ (x,z,t_1+t_2) : (x,y,t_1) \in S, (y,z,t_2) \in R \} \]

Given a convex bifunction $F : \R^m \cxto \R^n$, the epigraph of its graph function $\epi(\graph F) \subseteq \R^m \times \R^{n+1}$ is thus a Kleisli arrow for $T$. Under sufficient regularity assumptions, this is functorial, and we have an embedding $\keyword{epi} : \cxbifn \to \catname{CxRel}_T$.

\subsection{Ordinary Convex Programs}
We briefly discuss the historical origins of bifunctions in convex optimization \cite[\S~29-30]{rockafellar}: For simplicity, we say that a \emph{ordinary convex program} $P$ is a minimization problem of the form
\[ \inf \{ f(x) : x \in \R^n, g_1(x) \leq 0, \ldots, g_k(x) \leq 0 \} \]
where the objective function $f$ and the constraints $g_1, \ldots, g_k : \R^n \to \exR$ are finite convex functions. The \emph{bifunction associated to $P$} is defined as 
\[ F_P : \R^k \cxto \R^n, F_P(v,x) = f(x) + \sum_{i=1}^k \iv{f_i(x) \leq v_i} \]
The inputs of $v \in \R^k$ can be thought of as perturbations of the constraints. The so-called \emph{perturbation function} of $P$ is the parameterized minimization problem $(\inf F_P)(v) = \inff x {F_P(v,x)}$. The convex function $F_P(0,-)$ represents the unperturbed problem and $(\inf F_P)(0)$ is the desired solution. Note that in categorical language, the perturbation function is straightforwardly obtained as the bifunction composite $(\keyword{discard} \circ F_P) : \R^k \cxto I$, or graphically

\[ \begin{tikzpicture}[scale=0.3]
	\begin{pgfonlayer}{nodelayer}
		\node [style=none] (0) at (-3, 0) {};
		\node [style=morphism] (1) at (0, 0) {$F_P$};
		\node [style=bn] (2) at (3, 0) {};
		\node [style=none] (3) at (-5, 0) {$\R^k$};
	\end{pgfonlayer}
	\begin{pgfonlayer}{edgelayer}
		\draw (0.center) to (1);
		\draw (1) to (2);
	\end{pgfonlayer}
\end{tikzpicture}
 \]

\noindent The associated bifunction $F_P$ contains all information about the problem $P$, and allows one to find the dual problem $P^*$ by taking its adjoint. This way one can think of \emph{any} bifunction $\R^k \cxto \R^n$ as a \emph{generalized convex program} (\cite[p.~294]{rockafellar}).

\begin{example}[{\cite[p.~312]{rockafellar}}]
Consider a linear minimization problem $P$ of the form
\[ \inf \{ \langle c,x \rangle : b - Ax \leq 0 \} \]
The associated bifunction and its adjoint are
\begin{align*}
F(v,x) &= \langle c,x \rangle + \iv{x \geq 0, b-Ax \leq v } \\
F^*(x^*,v^*) &= \langle b, v^* \rangle - \iv{v^* \geq 0, c-A^Tv^* \geq x^*}
\end{align*}
which is the concave bifunction associated to the dual maximization problem
\[ \sup \{ \langle b,y \rangle : y \geq 0, c-A^Ty \geq 0 \} \]
\end{example}

\section{Gaussian Probability and Convexity}\label{sec:gaussians}

We now study the probabilistic applications of our categorical framework: Recently, a sizeable body of work in categorical probability theory has been developed in terms of copy-delete and Markov categories. A Markov category \cite{fritz2020synthetic} is a copy-delete category $(\mathbb C, \otimes, I)$ where every morphism $f : X \to Y$ is \emph{discardable} in the sense that $\keyword{discard}_Y \circ f = \keyword{discard}_X$. Classic examples of Markov categories are the category $\catname{FinStoch}$ of finite sets and stochastic matrices, and the category $\catname{Stoch}$ of measurable spaces and Markov kernels. Discardability expresses that probability measures are normalized (integrate to 1). Markov categories provide a natural semantic domain for probabilistic programs \cite{stein2021structural}.

In this section, we will focus on \emph{Gaussian probability}, by which we mean the study of multivariate normal (Gaussian) distributions and affine-linear maps. This is a small but expressive fragment of probability, which suffices for a range of interesting application from linear regression and Gaussian processes to Kalman filters. The univariate normal distribution $\N(\mu,\sigma^2)$ is defined on $\R$ via the density function
\[ f(x) = \frac{1}{\sqrt{2\pi\sigma^2}}\exp\left(-\frac{(x-\mu)^2}{2\sigma^2}\right)\]

Multivariate Gaussian distributions are easiest described as the laws of random vectors $A\cdot X + \mu$ where $A \in \R^{n \times k}$ and $X_1, \ldots, X_k \sim \N(0,1)$ are independent variables. The law is fully characterized by the mean $\mu$ and the covariance matrix $\Sigma = AA^T$. Conversely, for every vector $\mu \in \R^n$ and positive semidefinite matrix $\Sigma \in \R^{n \times n}$, there exists a unique Gaussian law $\N(\mu,\Sigma)$. If $X \sim \mathcal N(\mu,\Sigma)$ and $Y \sim \mathcal N(\mu',\Sigma')$ are independent then $X + Y \sim \N(\mu + \mu',\Sigma + \Sigma')$ and $AX \sim \N(A\mu,A\Sigma A^T)$. Gaussians are self-conjugate: If $(X,Y)$ are jointly Gaussian, then so is the conditional distribution $X|Y=y$ for any constant $y \in \R^k$.\\

If the covariance matrix $\Sigma$ is positive definite, then the Gaussian has a density with respect to the Lebesgue measure on $\R^n$ given by
\begin{equation}
     f(x) = \frac{1}{\sqrt{(2\pi)^n\det(\Sigma)}}\exp\left(-\frac 1 2 (x-\mu)^T\Omega(x-\mu)\right) \label{eq:gauss_density}
\end{equation}
where $\Omega = \Sigma^{-1}$ is known as the \emph{precision matrix}. This suggests two equivalent representations of Gaussians with different advantages (e.g. \cite{james:variance_manifold,stein2023gaussex}): 
\begin{itemize}
\item In \emph{covariance representation} $\Sigma$, pushforwards (addition, marginalization) are easy to compute. Conditioning requires solving an optimization problem
\item In \emph{precision representation} $\Omega$, conditioning is straightfoward. Pushforwards require solving an optimization problem.
\end{itemize}

If $\Sigma$ is singular, the Gaussian distribution is only supported on the affine subspace $\mu + S$ where $S = \im(\Sigma)$. In that case, the distribution has a density only with respect to the Lebesgue measure on $S$. This variability of base measure makes it complicated to work with densities, and by extension the precision representation.

The situation becomes clearer if we represent Gaussians by the quadratic functions induced by their covariance and precision matrices. These functions are convex (concave), and turn out to be adjoints of each other. This explains the duality of the two representations, and paves the way for generalizations of Gaussian probability like improper priors \cite{stein2023gaussex} which correspond to partial quadratic functions (\Sec\ref{sec:pqf}).

\subsection{Embedding Gaussians in Bifunctions}

We now give a categorical account of Gaussian probability (in covariance representation). A \emph{Gaussian morphism} $\R^m \to \R^n$ is a stochastic map of the form $x \mapsto Ax + \mathcal N(\mu,\Sigma)$, that is a linear map with Gaussian noise.

\begin{definition}[{\cite[\S 6]{fritz2020synthetic}}]\label{def:gauss}
The Markov prop $\gauss$ is given as follows
\begin{enumerate}
\item objects are the spaces $\R^n$, and $\R^m \otimes \R^n = \R^{m+n}$
\item morphisms $\R^m \to \R^n$ are tuples $(A,\mu,\Sigma)$ with $A \in \R^{n \times m}, \mu \in \R^n$ and $\Sigma \in \R^{n \times n}$ positive semidefinite
\item composition and tensor are given by the formulas
\begin{align*}
(A,\mu,\Sigma) \circ (B,\mu',\Sigma') &= (AB,\mu + A\mu',\Sigma + A\Sigma' A^T) \\
(A,\mu,\Sigma) \otimes (B,\mu',\Sigma') &= \left(A \oplus B, \mu \oplus \mu', \Sigma \oplus \Sigma')\right)
\end{align*}  
where $\oplus$ is block-diagonal composition.
\item the copy-delete structure is given by the linear maps $\Delta,!$
\end{enumerate}
\end{definition}

We have a Markov functor $\gauss \to \catname{Stoch}$ which sends $\R^n$ to the measurable space $(\R^n,\keyword{Borel}(\R^n))$ and assigns $(A,\mu,\Sigma)$ to the probability kernel given by $x \mapsto \N(Ax + \mu,\Sigma)$. Functoriality expresses that the formulas of Definition~\ref{def:gauss} agree with composition of Markov kernels given by integration of measures. Our main theorem shows that, surprisingly, the representation of Gaussians by quadratic functions is also functorial, i.e. we have an embedding $\gauss \to \cxbifn$.

\begin{theorem}\label{thm:gauss_functor}
We have functors of copy-delete categories in a duality situation
% https://q.uiver.app/#q=WzAsMyxbMSwwLCJcXGdhdXNzIl0sWzIsMiwiXFxjeGJpZm4iXSxbMCwyLCJcXGN2Ymlmbl5cXG9wIl0sWzAsMiwiXFxsb2dwZGYiLDJdLFswLDEsIlxcY2dmIl0sWzIsMSwiIiwyLHsiY3VydmUiOi0xLCJzdHlsZSI6eyJib2R5Ijp7Im5hbWUiOiJkYXNoZWQifX19XSxbMSwyLCIoLSleKiIsMix7ImN1cnZlIjotMSwic3R5bGUiOnsiYm9keSI6eyJuYW1lIjoiZGFzaGVkIn19fV1d
\[\begin{tikzcd}[ampersand replacement=\&]
	\& \gauss \\
	\\
	{(\cvbifn^\op,\black)} \&\& (\cxbifn,\white)
	\arrow["\logpdf"', from=1-2, to=3-1]
	\arrow["\cgf", from=1-2, to=3-3]
	\arrow[curve={height=-6pt}, dashed, from=3-1, to=3-3]
	\arrow["{(-)^*}"', curve={height=-6pt}, dashed, from=3-3, to=3-1]
\end{tikzcd}\]
The functors are defined as follows: Let $f = (A,\mu,\Sigma) \in \gauss(\R^m, \R^n)$, and define bifunctions
\begin{align*}
\logpdf_f : \R^n \cvto \R^m, \quad &\logpdf_f(y,x) = -\frac 1 2 \langle z,\Sigma^{-}z \rangle - \iv{z \in S} \\
\cgf_f : \R^m \cxto \R^n, \quad &\cgf_f(x,y) = \frac 1 2 \langle y,\Sigma y \rangle + \langle \mu, y \rangle + \iv{x = A^Ty}
\end{align*}
where $z = y-(Ax + \mu), S=\im(\Sigma)$ and $\Sigma^-$ denotes any \emph{generalized inverse} of $\Sigma$ (see \Sec\ref{appendix:partialquadratic}).
\end{theorem}
\begin{proof}
We elaborate the proof systematically in the appendix.
\end{proof}
The value $\logpdf_f(y,x)$ is indeed the conditional log-probability \eqref{eq:gauss_density} minus a scalar. The name $\keyword{cgf}$ is short for cumulant-generating function, which we elaborate in \Sec\ref{sec:cgf}. For now, we can see $\keyword{cgf}$ as a generalized covariance representation.

\subsection{Outlook: Gaussian Relations}\label{sec:pqf}

Measure-theoretically, there is no uniform probability distribution over the real line. Such a distribution, if it existed, would be useful to model complete absence of information about a point $X$ -- in Bayesian inference, this is called an uninformative prior. Intuitively, such a distribution should have density $1$, but this would not integrate to 1. On the other hand, a formal logdensity of $0$ makes sense -- this is simply the indicator function of the full subset $\R \subseteq \R$.

An \emph{extended Gaussian distribution}, as described in \cite{stein2023gaussex}, is a formal sum $\mathcal N(\mu,\Sigma) + D$ of a Gaussian distribution and a vector subspace $D \subseteq \R^n$, called a fibre, thereby blending relational and probabilistic nondeterminism. Such entities were considered by Willems in the control theory literature, under the name of linear open stochastic systems \cite{willems:constrained,willems:oss}; he identifies them with Gaussian distributions on the quotient space $\R^n/D$. A categorical account based on decorated cospans is developed in \cite{stein2023gaussex}. \\

It is straightforward to embed extended Gaussian distributions into convex bifunctions, by taking the sum of the interpretations from Theorems~\ref{thm:linrel_functor}~and~\ref{thm:gauss_functor}. The distribution $\psi = \mathcal N(\mu,\Sigma) + D$ has a convex interpretation given by
\[ \cgf_\psi(x) = \frac 1 2 \langle x,\Sigma x \rangle + \langle \mu,x \rangle + \iv{x \in D^\bot} \]
Functions of this form are \emph{partial convex quadratic functions}, which are known to form a well-behaved class of convex functions (see \Sec\ref{app:pcqf}). The theory of such functions can be understood as an extension of Gaussian probability with relational nondeterminism and conditioning, which we term \emph{Gaussian relations}. In Gaussian relations, we achieve full symmetry between covariance and density representation (that is, there exists a color-swap symmetry). 

Partiality is necessary to be able to interpret all generators of \eqref{eq:gla_generators}; on the upside, the presence of partiality makes conditioning a first-class operation: For example, if $f : \R^2 \cvto I$ is the joint logdensity of Gaussian variables $(X,Y)$, then conditioning on $Y=0$ is the same as computing the bifunction composite with the $\keyword{zero}$ map, which is a simple restriction of logdensity $f_{X|Y=0}(x) = f(x,0)$. On the other hand, conditioning in the covariance representation $f^*$ requires solving the infimization problem $\inff {y^*} {f^*(x^*,y^*)}$. Graphically, we have
\[ \begin{tikzpicture}[scale=0.3]
	\begin{pgfonlayer}{nodelayer}
		\node [style=morphism] (0) at (-5.75, 0) {$f$};
		\node [style=none] (1) at (-10, 0.5) {};
		\node [style=wn] (2) at (-8, -0.5) {};
		\node [style=none] (3) at (-6.25, 0.5) {};
		\node [style=none] (4) at (-6.25, -0.5) {};
		\node [style=morphism] (5) at (2.75, 0) {$f^*$};
		\node [style=none] (6) at (3.25, 0.5) {};
		\node [style=none] (7) at (3.25, -0.5) {};
		\node [style=none] (8) at (8, 0.5) {};
		\node [style=bn] (9) at (5, -0.5) {};
		\node [style=none] (10) at (-1.5, 0) {vs.};
	\end{pgfonlayer}
	\begin{pgfonlayer}{edgelayer}
		\draw (2) to (4.center);
		\draw (3.center) to (1.center);
		\draw (7.center) to (9);
		\draw (8.center) to (6.center);
	\end{pgfonlayer}
\end{tikzpicture}
 \]

\section{A Wider Perspective}\label{sec:perspective}

The example of Gaussian probability was particular situation in which we could map probabilistic concepts to concepts of convex analysis in a functorial way. In this section, we will take an even wider perspective and view convex bifunctions as a categorical model of probability on its own. We will then point out known connections between probability theory and convex analysis, such as the Laplace approximation and cumulant generating functions.

\subsection{The Laplace Approximation}

For every copy-delete category $\mathbb C$, the subcategory of discardable morphisms is a Markov category, and can therefore be seen as a generalized model of probability theory. We investigate this notion for categories of bifunctions.

\begin{proposition}\label{prop:normalization}
Let $F : \R^m \cxto \R^n, G : \R^n \cvto \R^m$ be bifunctions. Then
\begin{enumerate}
\item $F$ is discardable in $(\cxbifn,\black)$ if $\forall x, \inf_y F(x,y) = 0$ 
\item $G$ is discardable in $(\cvbifn^\op,\white)$ if $\forall x, G(0,x) = \iv{x=0}$
\end{enumerate}
and the adjoint $(-)^*$ defines a bijection between the two.
\end{proposition}
\begin{proof}
Direct calculation.
\end{proof}
The embedding of Theorem~\ref{thm:gauss_functor} takes values in discardable bifunctions and hence preserve Markov structure. Functoriality means that the composition of Gaussians (integration) and the composition of bifunctions (optimization) coincide. For general probability distributions, this will no longer be the case. We can however understand bifunction composition as an approximation of ordinary probability theory under the so-called \emph{Laplace approximation}. In its simplest instance, Laplace's method (or method of steepest ascent) is a method to approximate certain integrals by finding the maxima of its integrand (e.g. \cite{touchette2005legendre})
\[ \int e^{n(\langle c,x \rangle - f(x))} \mathrm d x \approx \exp\left(n\supp x {\langle c,x\rangle - f(x)}\right) \text{ for } n \to \infty \]

A wide class of commonly used probability distributions is \emph{log-concave}, including Gaussian, Laplace, Dirichlet, exponential and uniform distributions. Laplace's approximation (e.g. \cite[\S 27]{mackay2003information}) is a way of approximating such distributions around their mode $x_0$ by a normal distribution, as the Taylor expansion of their logdensity resembles a Gaussian one
\[ h(x) \approx h(x_0) + \frac 1 2 h''(x_0)(x-x_0)^2 \]
We can attempt to reduce questions about such distributions to mode-finding (maximization). The Laplace approximation is fundamental in many applications such as neuroscience \cite{friston2009predictive,friston2007variational} and has been generalized to a large body of literature on so-called saddle-point methods \cite{butler2007saddlepoint,mccullagh2018tensor}. The existence of the functor from Gaussians to bifunctions expresses that, as desired, the Laplace approximation is exact on Gaussians. We give an example of the approximation \emph{not} being exact (ironically) on Laplacian distributions.

\begin{example}
The standard Laplacian distribution has the density function $f(x) = \frac 1 2 \exp(|x|)$ on the real line. The logpdf $h(x) = |x|$ is a convex function whose convex conjugate is given by $h^*(x^*) = \iv{|x^*| \leq 1}$ (see Example~\ref{prop:ex_abs}). The latter function is idempotent under addition, and conversely $h \conv h = h$, so $h$ is idempotent under infimal convolution. In contrast, the density $f(x)$ is not idempotent under integral convolution: The sum of independent standard Laplacians is \emph{not} itself Laplacian.
\end{example}

\subsection{Convex Analysis in Probability Theory}\label{sec:cgf}
For a random variable $X$ on $\R^n$, the moment generating function $M_X$ is defined by the following expectation (provided that it exists) $M_X(x^*) = \mathbb E[e^{\langle x^*,X\rangle}]$. The \emph{cumulant-generating function} is defined as its logarithm $c_X(x^*) = \log M_X(x^*)$. The function $c_X$ is always convex. The cumulant-generating function of a multivariate Gaussian $X \sim \mathcal N(\mu,\Sigma)$ is precisely
\begin{equation} c_X(x^*) = \frac 1 2 \langle x^*,\Sigma x^* \rangle + \langle x^*,\mu \rangle \label{eq:gauss_cgf} \end{equation}
which explains our choice of the convex bifunction $\keyword{cgf}$ associated to a Gaussian morphism in Theorem~\ref{thm:gauss_functor}. The notion of cumulant-generating function has a central place in the study of exponential families.

It is a particular fact about Gaussians that the cumulant-generating function is the convex conjugate of the logdensity. In the general case, the convex conjugate $c_X^*(x)$ does have a probabilistic interpretations as a so called-rate function in large deviations theory (Cram\'er's theorem, \cite{cramer1938nouveau}). It has also been used to formulate a variational principle \cite{zajkowski2017variational}.

\subsection{Idempotent Mathematics}

We zoom out to an even wider perspective: This subsection briefly outlines some further background of the connections between convex and probabilistic world: The logarithm of base $t<1$ defines an isomorphism of semirings $([0,\infty), \times, +) \to (\R \cup \{+\infty\}, +, \oplus_t)$ where $\oplus_t$ is $x \oplus_t y = \log_t(t^x + t^y)$. In the `tropical limit' $t \searrow 0$, we have $x \oplus_t y \approx \min(x,y)$, so we can consider working in the semiring $(\exR, +, \min)$ as a limit or deformation of the usual operations on the reals. The semiring $\exR$ is idempotent, meaning $x \oplus x= \min(x,x) = x$, hence this field of study is also known as \emph{idempotent mathematics} \cite{puhalskii2001largedeviations}, and the limiting procedure has been called \emph{Maslov dequantization} \cite{litvinov2007}. Our definition of convex bifunctions in terms of the idempotent semiring $\exR$ thus carries a strong flavor of idempotent mathematics. 

Idempotent analogues of measure theory are discussed in \cite{puhalskii2001largedeviations,litvinov2007}, and many theorems in classical probability theory are mirrored by theorems of idempotent probability theory. For example, the idempotent analogue of integration is infimization; under this view, the tropical analogue of the Laplace transform (cf. moment-generating function) is the Legendre transform \cite[\S 7]{litvinov2007}
\[ \int e^{\langle x^*, x \rangle} f(x) \mathrm dx \quad\leftrightarrow\quad \inff x {\langle x^*, x \rangle + f(x)} \]
which explains the appearance of the cumulant-generating function in our work. Theorem~\ref{thm:gauss_functor} means that for Gaussians, it makes no difference whether we work in the real-analytic or idempotent world. Idempotent Gaussians have been defined in \cite[1.11.10]{puhalskii2001largedeviations} using the same formula \eqref{eq:gauss_cgf}.

\section{Related and Future Work}\label{sec:related}

We have described categories of bifunctions as a compositional setting for convex analysis which subsumes a variety of formalisms like linear functions and relations, as well as convex optimization problems, and has a rich duality theory and an elegant graphical language. We have then explored connections between convex analysis and probability theory, and showed that Gaussian probability can be equivalently described in a measure-theoretic and a convex-analytic language. The equivalence of these two perspectives is elegantly formalized as a structure-preserving functor between copy-delete categories. It will be interesting to see how this approach can be generalized to larger classes of distributions such as exponential families.

Concurrently to our work, the categorical structure of convex bifunctions has been exploited by \cite{hanks2023compositional} to compositionally build up objective functions for MPC in control theory. That work does not explore Legendre duality and the connections with categorical models of probability theory. The language of props has a history of applications in engineering \cite{baez:control,baez:props,bonchi:cat_signal_flow}, and our work was directly inspired by the semantics of probabilistic programming \cite{stein2021compositional,stein2021structural}.

A starting point for future work is to flesh out the outlook given in \Sec\ref{sec:pqf}, that is to define a hypergraph category of partial quadratic convex functions, which generalizes Gaussian and extended Gaussian probability. It is also interesting to give a presentation for this prop in the style of \cite{graphical_la}: We believe that this is achieved by the addition of a single generator $\nu : I \to \R$ to graphical affine algebra \cite{pbsz} which represents the quadratic function $f(x) = \frac 1 2 x^2$, and that its equational theory is essentially given by invariance under the orthogonal groups $O(n)$. A similar equational theory has been attempted in \cite{stein2021compositional} though no completeness has been proven. Diagrammatic presentations of concepts from geometry and optimization such as polyhedral algebra and Farkas lemma have been given in \cite{bonchi_et_al:polyhedral,bonchi_et_al:farkas}. 

We realize that the dependence on regularity assumptions (the caveat of \Sec\ref{sec:intro_cx}) makes general theorems about categories of bifunctions like Theorem~\ref{thm:adjoint_functoriality} somewhat awkward to state. We still believe that using a general categorical language is a useful way of structuring the field and making connections, but see the following avenues of improving the technical situation
\begin{enumerate}
\item Identifying specific, well-behaved subcategories of bifunctions (such as convex relations, (partial) linear and (partial) quadratic functions) on which everything behaves as desired. This was pursued in \Sec\ref{sec:examples} and \Sec\ref{sec:gaussians}.
\item The Legendre-Fenchel transform has been phrased in terms of enriched adjunctions in \cite{willerton}. It stands to hope that developing this enriched-categorical approach may take care of some regularity conditions in a systematic way.
\end{enumerate}

\subsubsection*{Acknowledgements}
We thank the anonymous reviewers for their careful reviews and suggestions for this work. 

\bibliographystyle{splncs04}
\bibliography{references}

\begin{thebibliography}{10}
\providecommand{\url}[1]{\texttt{#1}}
\providecommand{\urlprefix}{URL }
\providecommand{\doi}[1]{https://doi.org/#1}

\bibitem{baez:props}
Baez, J.C., Coya, B., Rebro, F.: Props in network theory (2018)

\bibitem{baez:control}
Baez, J.C., Erbele, J.: Categories in control. Theory Appl.~Categ.  \textbf{30},  836--881 (2015)

\bibitem{bolt2019}
Bolt, J., Coecke, B., Genovese, F., Lewis, M., Marsden, D., Piedeleu, R.: Interacting conceptual spaces i: Grammatical composition of concepts. Conceptual spaces: Elaborations and applications pp. 151--181 (2019)

\bibitem{bonchi_et_al:polyhedral}
Bonchi, F., Di~Giorgio, A., Soboci\'{n}ski, P.: {Diagrammatic Polyhedral Algebra}. In: Boja\'{n}czyk, M., Chekuri, C. (eds.) 41st IARCS Annual Conference on Foundations of Software Technology and Theoretical Computer Science (FSTTCS 2021). Leibniz International Proceedings in Informatics (LIPIcs), vol.~213, pp. 40:1--40:18. Schloss Dagstuhl -- Leibniz-Zentrum f{\"u}r Informatik, Dagstuhl, Germany (2021). \doi{10.4230/LIPIcs.FSTTCS.2021.40}, \url{https://drops.dagstuhl.de/entities/document/10.4230/LIPIcs.FSTTCS.2021.40}

\bibitem{bonchi_et_al:farkas}
Bonchi, F., Di~Giorgio, A., Zanasi, F.: {From Farkas' Lemma to Linear Programming: an Exercise in Diagrammatic Algebra}. In: Gadducci, F., Silva, A. (eds.) 9th Conference on Algebra and Coalgebra in Computer Science (CALCO 2021). Leibniz International Proceedings in Informatics (LIPIcs), vol.~211, pp. 9:1--9:19. Schloss Dagstuhl -- Leibniz-Zentrum f{\"u}r Informatik, Dagstuhl, Germany (2021). \doi{10.4230/LIPIcs.CALCO.2021.9}, \url{https://drops.dagstuhl.de/entities/document/10.4230/LIPIcs.CALCO.2021.9}

\bibitem{pbsz}
Bonchi, F., Piedeleu, R., Sobocinski, P., Zanasi, F.: Graphical affine algebra. In: Proc.~LICS 2019 (2019)

\bibitem{bonchi:cat_signal_flow}
Bonchi, F., Soboci{\'n}ski, P., Zanasi, F.: A categorical semantics of signal flow graphs. In: CONCUR 2014--Concurrency Theory: 25th International Conference, CONCUR 2014, Rome, Italy, September 2-5, 2014. Proceedings 25. pp. 435--450. Springer (2014)

\bibitem{bsz}
Bonchi, F., Sobocinski, P., Zanasi, F.: The calculus of signal flow diagrams {I}: linear relations on streams. Inform.~Comput.  \textbf{252} (2017)

\bibitem{bonchi:interacting_hopf}
Bonchi, F., Soboci{\'n}ski, P., Zanasi, F.: Interacting {H}opf algebras. Journal of Pure and Applied Algebra  \textbf{221}(1),  144--184 (2017)

\bibitem{butler2007saddlepoint}
Butler, R.W.: Saddlepoint approximations with applications, vol.~22. Cambridge University Press (2007)

\bibitem{cho_jacobs}
Cho, K., Jacobs, B.: Disintegration and {Bayesian} inversion via string diagrams. Mathematical Structures in Computer Science  \textbf{29},  938 -- 971 (2019)

\bibitem{coecke2018}
Coecke, B., Genovese, F., Lewis, M., Marsden, D., Toumi, A.: Generalized relations in linguistics \& cognition. Theoretical Computer Science  \textbf{752},  104--115 (2018). \doi{https://doi.org/10.1016/j.tcs.2018.03.008}, \url{https://www.sciencedirect.com/science/article/pii/S0304397518301476}, quantum structures in computer science: language, semantics, retrieval

\bibitem{cramer1938nouveau}
Cram{\'e}r, H.: Sur un nouveau theoreme-limite de la theorie des probabilities. Scientifiques et Industrielles  \textbf{736},  5--23 (1938)

\bibitem{fong2019hypergraph}
Fong, B., Spivak, D.I.: Hypergraph categories. Journal of Pure and Applied Algebra  \textbf{223}(11),  4746--4777 (2019)

\bibitem{friston2009predictive}
Friston, K., Kiebel, S.: Predictive coding under the free-energy principle. Philosophical transactions of the Royal Society B: Biological sciences  \textbf{364}(1521),  1211--1221 (2009)

\bibitem{friston2007variational}
Friston, K., Mattout, J., Trujillo-Barreto, N., Ashburner, J., Penny, W.: Variational free energy and the laplace approximation. Neuroimage  \textbf{34}(1),  220--234 (2007)

\bibitem{fritz2020synthetic}
Fritz, T.: A synthetic approach to markov kernels, conditional independence and theorems on sufficient statistics. Advances in Mathematics  \textbf{370},  107239 (2020)

\bibitem{fujii2019}
Fujii, S.: A categorical approach to l-convexity. arXiv preprint arXiv:1904.08413  (2019)

\bibitem{hanks2023compositional}
Hanks, T., She, B., Hale, M., Patterson, E., Klawonn, M., Fairbanks, J.: A compositional framework for convex model predictive control. arXiv preprint arXiv:2305.03820  (2023)

\bibitem{james:variance_manifold}
JAMES, A.: The variance information manifold and the functions on it. In: Multivariate Analysis–III, pp. 157--169. Academic Press (1973). \doi{https://doi.org/10.1016/B978-0-12-426653-7.50016-8}, \url{https://www.sciencedirect.com/science/article/pii/B9780124266537500168}

\bibitem{litvinov2007}
Litvinov, G.L.: Maslov dequantization, idempotent and tropical mathematics: A brief introduction. Journal of Mathematical Sciences  \textbf{140},  426--444 (2007)

\bibitem{mackay2003information}
MacKay, D.J.: Information theory, inference and learning algorithms. Cambridge university press (2003)

\bibitem{marsden2017}
Marsden, D., Genovese, F.: Custom hypergraph categories via generalized relations. arXiv preprint arXiv:1703.01204  (2017)

\bibitem{mccullagh2018tensor}
McCullagh, P.: Tensor methods in statistics. Courier Dover Publications (2018)

\bibitem{graphical_la}
Paixão, J., Rufino, L., Sobociński, P.: High-level axioms for graphical linear algebra. Science of Computer Programming  \textbf{218},  102791 (2022). \doi{https://doi.org/10.1016/j.scico.2022.102791}, \url{https://www.sciencedirect.com/science/article/pii/S0167642322000247}

\bibitem{puhalskii2001largedeviations}
Puhalskii, A.: Large deviations and idempotent probability. CRC Press (2001)

\bibitem{rockafellar}
Rockafellar, R.T.: Convex Analysis, vol.~11. Princeton University Press (1997)

\bibitem{selinger:graphical}
Selinger, P.: A Survey of Graphical Languages for Monoidal Categories, pp. 289--355. Springer Berlin Heidelberg, Berlin, Heidelberg (2011). \doi{10.1007/978-3-642-12821-9_4}

\bibitem{selinger2007dagger}
Selinger, P.: Dagger compact closed categories and completely positive maps. Electronic Notes in Theoretical computer science  \textbf{170},  139--163 (2007)

\bibitem{stein2021structural}
Stein, D.: Structural foundations for probabilistic programming languages. University of Oxford  (2021)

\bibitem{stein2023gaussex}
Stein, D., Samuelson, R.: A category for unifying gaussian probability and nondeterminism. In: 10th Conference on Algebra and Coalgebra in Computer Science (CALCO 2023). Schloss Dagstuhl-Leibniz-Zentrum f{\"u}r Informatik (2023)

\bibitem{stein2021compositional}
Stein, D., Staton, S.: Compositional semantics for probabilistic programs with exact conditioning. In: 2021 36th Annual ACM/IEEE Symposium on Logic in Computer Science (LICS). pp. 1--13. IEEE (2021)

\bibitem{touchette2005legendre}
Touchette, H.: Legendre-fenchel transforms in a nutshell. Unpublished Report (Queen Mary University of London)  (2005)

\bibitem{willems:constrained}
Willems, J.C.: Constrained probability. In: 2012 IEEE International Symposium on Information Theory Proceedings. pp. 1049--1053 (2012). \doi{10.1109/ISIT.2012.6283011}

\bibitem{willems:oss}
Willems, J.C.: Open stochastic systems. IEEE Transactions on Automatic Control  \textbf{58}(2),  406--421 (2013). \doi{10.1109/TAC.2012.2210836}

\bibitem{willerton}
Willerton, S.: The {L}egendre-{F}enchel transform from a category theoretic perspective. arXiv preprint arXiv:1501.03791  (2015)

\bibitem{zajkowski2017variational}
Zajkowski, K.: A variational formula on the cram{\'e}r function of series of independent random variables. Positivity  \textbf{21}(1),  273--282 (2017)

\end{thebibliography}

\section*{Appendix}

\subsection{Functoriality for $\gauss$}\label{app:gauss}

We elaborate the proof of Theorem~\ref{thm:gauss_functor}:

\begin{proposition}
The assignment $\cgf : \gauss \to \cxbifn$ is functorial.
\end{proposition}
\begin{proof}
Let $f=(A,a,\Sigma) : \R^n \to \R^k$ and $g=(B,b,\Xi) : \R^m \to \R^n$ be two morphisms in $\gauss$, and
\begin{align*}
\cgf_f(y,z) &= \frac 1 2 \langle z,\Sigma z \rangle + \langle a, z \rangle + \iv{y=A^Tz} \\
\cgf_g(x,y) &= \frac 1 2 \langle y,\Xi y \rangle + \langle b, y \rangle + \iv{x=B^Ty} 
\end{align*}
their denotations, then it is straightforward to check that
\begin{align*}
&(\cgf_f \circ \cgf_g)(x,z) \\
=& \inff y {\frac 1 2 \langle z,\Sigma z \rangle + \langle a, z \rangle + \iv{y=A^Tz} + \frac 1 2 \langle y,\Xi y \rangle + \langle b, y \rangle + \iv{x=B^Ty} } \\
=& \frac 1 2 \langle z,\Sigma z \rangle + \langle a, z \rangle + \frac 1 2 \langle A^Tz,\Xi A^Tz \rangle + \langle b, A^Tz \rangle + \iv{x=B^TA^Tz} \\
=& \frac 1 2 \langle z,\Sigma z \rangle + \langle a, z \rangle + \frac 1 2 \langle z,A\Xi A^Tz \rangle + \langle Ab, z \rangle + \iv{x=(AB)^Tz} \\
=& \frac 1 2 \langle z,(\Sigma + A\Xi A^T) z \rangle + \langle a+Ab, z \rangle + \iv{x=(AB)^Tz} \\
=& \cgf_{f\circ g}(x,z)
\end{align*}
\end{proof}

\subsection{Duality Theory of Quadratic Functions}\label{appendix:partialquadratic}
We spell out some calculations involving the adjoints of quadratic forms (as elaborated in \cite[p.~107]{rockafellar}).

\begin{proposition}
Let $A \in \R^{m \times n}$ be any symmetric matrix and $S=\im(A)$. Then the restricted linear map $S \to S, x \mapsto Ax$ is invertible, meaning $\forall y \in S \exists! x \in S, Ax=y$.
\end{proposition}
\begin{proposition}
Easy consequence of $\ker(A) = \im(A^\bot)^\bot = \im(A)^\bot$.
\end{proposition}

Let $A \in \R^{m \times n}$ be any matrix. A \emph{generalized inverse} of $A$ is a matrix $A^\g \in \R^{n \times m}$ such that $AA^\g A=A$. This means that for any $y \in \im(A)$, the vector $x=A^\g y$ must satisfy $Ax=y$. Generalized inverses exist for every matrix, but they are generally not unique unless $A$ is invertible, in which case $A^\g=A^{-1}$. A particularly canonical generalized inverse is the Moore-Penrose pseudoinverse $A^+$, but our development does not rely on this particular choice. 

\begin{proposition}\label{prop:gen_inverse}
Let $A \in \R^{n \times n}$ be a symmetric matrix, then the value $\langle y,A^\g y\rangle$ does not depend on the choice of generalized inverse $A^\g$ for all $y\in \im(S)$.
\end{proposition}
\begin{proof}
Let $y=Ax$, then $\langle Ax,A^\g Ax \rangle = \langle x,AA^\g A x \rangle = \langle x,A x \rangle$, so we need to show that that value does not depend on the choice of $x$. Let $x'$ be another solution, then $A(x-x')=0$, and we can derive $\langle x,Ax \rangle - \langle x',Ax' \rangle = 0$.
\end{proof}

\begin{proposition}
Let $\Sigma \in \R^{n \times n}$ be a positive definite matrix, and consider the convex quadratic function $f(x) = \frac 1 2\langle x,\Sigma x\rangle$, then $f^*(x^*) = \frac 1 2 \langle x^*, \Sigma^{-1} x^* \rangle$.
\end{proposition}
\begin{proof}
The function $x \mapsto \langle x^*,x \rangle - f(x)$ is differentiable with gradient $x^*-\Sigma x$, hence its minimum is attained at $x=\Sigma^{-1}x^*$. 
\end{proof}

\begin{proposition}\label{ref:conj_semidef}
Let $\Sigma \in \R^{n \times n}$ be a positive \emph{semidefinite} matrix and let $f(x) = \frac 1 2\langle x,\Sigma x\rangle$. Then
\[ f^*(x^*) = \frac 1 2 \langle x^*, \Sigma^\g x^* \rangle + \iv{x^* \in S} \]
where $S=\im(\Sigma)$ and $\Sigma^\g$ is any generalized inverse of $\Sigma$.
\end{proposition}
\begin{proof}
We consider the supremum
\[ f^*(x^*) = \supp x {\langle x^*,x \rangle - \frac 1 2 \langle x,\Sigma x \rangle}\]
If $x^* \notin S$, the supremum is $+\infty$; we proceed with the remaining case $x^* \in S$. Then in fact, we can restrict the computation to
\begin{equation} \supp {x \in S} {\langle x^*,x \rangle - \frac 1 2 \langle x,\Sigma x \rangle} \label{eq:restricted_sup} \end{equation}
because components orthogonal to $S$ don't change the value of the supremand. The restricted map $S \to S, x \mapsto \Sigma x$ is invertible, so we can find a matrix $T$ such that $\forall y \in S, Ty \in S$ and $\Sigma T y = T\Sigma y = y$. By the previous proposition, the desired convex conjugate becomes
\[ f^*(x^*) = \frac 1 2 \langle x^*,Tx^* \rangle + \iv{x^* \in S} \]
We note that $T$ is a generalized inverse of $\Sigma$, and by Proposition~\ref{prop:gen_inverse}, any generalized inverse will work in place of $T$.
\end{proof}

\begin{proposition}
$\cgf$ and $\logpdf$ are adjoints for Gaussians.
\end{proposition}
\begin{proof}
Let $f = (A,\mu,\Sigma) \in \gauss(\R^m,\R^n)$. We show that $\cgf_f^* = \logpdf$. The converse follows because partial quadratic functions are closed and satisfy the hypothesis of Prop.~\ref{prop:double_conjugate}. We have
\begin{align*}
\cgf_f^*(y^*,x^*) &= \inff {x,y} {\frac 1 2 \langle y,\Sigma y \rangle + \langle \mu, y \rangle + \iv{x=A^Ty} + \langle x^*,x\rangle - \langle y^*,y \rangle} \\
&= \inff y {\frac 1 2 \langle y,\Sigma y \rangle + \langle \mu, y \rangle + \langle x^*,A^Ty\rangle - \langle y^*,y \rangle} \\
&= \inff y {\frac 1 2 \langle y,\Sigma y \rangle - \langle y,z \rangle}
\end{align*}
where $z=y^*-(Ax^*+\mu)$. If $\Sigma$ is invertible, the gradient with respect to $y$ of the infimand is $\Sigma y - z$ and the minimum is attained for $y = \Sigma^{-1}z$. The optimal value is
\begin{align*}
\cgf_f^*(y^*,x^*) &= \frac 1 2 \langle \Sigma^{-1}z,\Sigma\Sigma^{-1}z \rangle - \langle y, \Sigma^{-1}z \rangle = -\frac 1 2 \langle z,\Sigma^{-1}z \rangle
\end{align*}
If $\Sigma$ is singular, we use the generalized inverse $\Sigma^-$ of Prop.~\ref{ref:conj_semidef} together with the condition $\iv{y^* \in \im(\Sigma)}$, giving the formula $\logpdf_f$ of Theorem~\ref{thm:linrel_functor} as desired.
\end{proof}

\subsection{Partial Quadratic Functions}\label{app:pcqf}
The theory of (partial) quadratic functions is spelled out in \cite[p.~109]{rockafellar}, which we summarize here: 
A quadratic function $q : \R^n \to \R$ is convex if and only if it is of the form
\[ q(x) = \langle x, A x \rangle + \langle \mu,x \rangle + c \]
with $A$ positive semidefinite. A \emph{partial convex quadratic function} (pcqf) is function of the form
\[ f(x) = q(x) + \iv{x \in M} \]
where $q$ is a convex quadratic function and $M \subseteq \R^n$ is an affine subspace. One can show that every pcqf arises via suitable linear transformations from an \emph{elementary pcqf} given by the diagonal form
\[  h(x) = \frac 1 2 (\lambda_1x_1^2 + \ldots + \lambda_nx_n^2) \text{ with } \lambda_i \in [0,+\infty] \]
Its domain is the subspace $\{ x : \forall i, \lambda_i = +\infty \Rightarrow x_i = 0 \}$.
Its conjugate is of the same form
\[ h^*(x) = \frac 1 2 (\lambda_1^*x_1^2 + \ldots + \lambda_n^*x_n^2) \text{ with } \lambda_i \in [0,+\infty] \]
where 
\[ \lambda^* = \begin{cases}
    +\infty, &\lambda = 0 \\
    0, &\lambda = +\infty \\
    \lambda^{-1}, &\text{otherwise}
\end{cases}\]
From this formula, we can derive that the class of pcqf is this closed under convex conjugation.

\end{document}